\documentclass[reqno]{amsart}
\usepackage{float}
\usepackage{amsmath}
\usepackage{graphicx}
\usepackage{latexsym}
\usepackage{amsfonts}
\usepackage{amssymb}

\theoremstyle{plain}
\newtheorem{theorem}{Theorem}
\newtheorem{corollary}[theorem]{Corollary}
\newtheorem{lemma}[theorem]{Lemma}

\theoremstyle{definition}

\newtheorem{remark}[theorem]{Remark}
\newtheorem*{remark*}{Remark}

\begin{document}
\title[Sudden Extinction of critical BPRE]{ Sudden extinction of a critical
branching process in random environment}
\thanks{Supported in part by the DFG-RFBR grant 08-01-91954}
\author[Vatutin]{Vladimir A. Vatutin}
\address{Steklov Mathematical Institute RAS, Gubkin street 8, 19991, Moscow
\\
Russia}
\email{vatutin@mi.ras.ru}
\author[Wachtel]{Vitali Wachtel}
\address{ Bereich M5, Technische Universit{\"{a}}t M{\"{u}}nchen, Zentrum
Mathematik, D-85747 Garching bei M{\"{u}}nchen}
\email{wachtel@ma.tum.de}
\date{\today }

\begin{abstract}
Let $T$ be the extinction moment of a critical branching process $Z=\left(
Z_{n},n\geq 0\right) $ in a random environment specified by iid probability
generating functions. We study the asymptotic behavior of the probability of
extinction of the process $Z$ at moment $n\rightarrow \infty ,$ and show
that if the logarithm of the (random) expectation of the offspring number
belongs to the domain of attraction of a non-gaussian stable law then the
extinction occurs owing to very unfavorable environment forcing the process,
having at moment $T-1$ exponentially large population, to die out. We also
give an interpretation of the obtained results in terms of random walks in
random environment.

\end{abstract}

\maketitle

\section{Introduction and mains results}

We consider a branching process in a random environment specified by a
sequence of independent identically distributed random offspring generating
functions
\begin{equation}
f_{n}(s):=\sum_{k=0}^{\infty }f_{nk}s^{k},\quad n\geq 0.  \label{DefF}
\end{equation}%
Denoting by $Z_{n}$ the number of particles in the process at time $n$ we
define it's evolution by the relations
\begin{align*}
& Z_{0}:=1, \\
& \mathbf{E}[s^{Z_{n+1}}|f_{0},f_{1},\ldots ,f_{n};Z_{0},Z_{1},\ldots
,Z_{n}]:=(f_{n}(s))^{Z_{n}},\quad n\geq 0.
\end{align*}

Put $X_{k}:=\log f_{k-1}^{\prime }(1)$, $k\geq 1,$ and denote $S_{0}:=0$, $%
S_{n}:=X_{1}+X_{2}+\ldots +X_{n}$. Following \cite{AGKV05} we call the
process $Z:=\left\{ Z_{n},\,n\geq 0\right\} $ \textit{critical} if and only
if the random walk $S:=\left\{ S_{n},n\geq 0\right\} $ is oscillating, that
is,
\begin{equation*}
\limsup_{n\rightarrow \infty }S_{n}=\infty \ \text{ and }\
\liminf_{n\rightarrow \infty }S_{n}=-\infty
\end{equation*}%
with probability $1$. This means that the stopping time
\begin{equation*}
T^{-}:=\min \{k\geq 1:S_{k}<0\}
\end{equation*}%
is finite with probability $1$ and, as a result (see \cite{AGKV05}), the
extinction moment
\begin{equation*}
T:=\min \{k\geq 1:Z_{k}=0\}
\end{equation*}%
of the process $Z$ is finite with probability $1$. For this reason it is
natural to study the asymptotic behavior of the survival probability $%
\mathbf{P}(T>n)$ as $n\rightarrow \infty .$ This has been done in \cite%
{AGKV05}: If
\begin{equation}
\lim_{n\rightarrow \infty }\mathbf{P}\left( S_{n}>0\right) =:\rho \in (0,1),
\label{Spit}
\end{equation}%
then (under some mild additional assumptions)
\begin{equation}
\mathbf{P}(T>n)\sim \theta \mathbf{P}(T^{-}>n)=\theta \frac{l(n)}{n^{1-\rho }%
},  \label{tails}
\end{equation}%
where $l(n)$ is a slowly varying function and $\theta $ is a known positive
constant whose explicit expression is given by formula (4.10) in \cite%
{AGKV05}.\footnote{%
We write $a_{n}\sim b_{n}$ if $\lim_{n\rightarrow \infty }\left(
a_{n}/b_{n}\right) =1$.}

A local version of (\ref{tails}) was obtained in \cite{VD97}, where it was
established that if the offspring generating functions $f_{n}(s),n=0,1,...,$
are fractional-linear with probability~$1$ and (along with some other
conditions) $\mathbf{E}X_{n}=0$ and $VarX_{n}\in (0,\infty )$ then
\begin{equation}
\mathbf{P}(T=n)\sim \theta \mathbf{P}(T^{-}=n)\sim \frac{C}{n^{3/2}}.
\label{vdd}
\end{equation}

The aim of the present paper is to refine equivalence (\ref{tails}) and to
complement (\ref{vdd}) by investigating the asymptotic behavior of the
probability $\mathbf{P}(T=n)$ as $n\rightarrow \infty $ in the case $%
VarX_{n}=\infty $. In addition, we consider the asymptotic behavior of the
joint distribution of the random variables $T$ and $Z_{T-1}$.

Let
\begin{equation*}
\mathcal{A}:=\{0<\alpha <1;\,|\beta |<1\}\cup \{1<\alpha <2;|\beta |\leq
1\}\cup \{\alpha =1,\beta =0\}\cup \{\alpha =2,\beta =0\}
\end{equation*}%
be a subset in $\mathbb{R}^{2}.$ For $(\alpha ,\beta )\in \mathcal{A}$ and a
random variable $X$ we write $X\in \mathcal{D}\left( \alpha ,\beta \right) $
if the distribution of $X$ belongs to the domain of attraction of a stable
law with characteristic function%
\begin{equation}
G_{\alpha ,\beta }\mathbb{(}t\mathbb{)}:=\exp \left\{ -c|t|^{\,\alpha
}\left( 1-i\beta \frac{t}{|t|}\tan \frac{\pi \alpha }{2}\right) \right\} ,\
c>0,  \label{std}
\end{equation}%
and, in addition, $\mathbf{E}X=0$ if this moment exists. Hence, there exists
a sequence $\left\{ c_{n},n\geq 1\right\} $ such that $c_{n}^{-1}S_{n}$
converges in distribution to the stable law whose with characteristic
function is specified by (\ref{std}). Observe that if $X_{n}\overset{d}{=}%
X\in \mathcal{D}\left( \alpha ,\beta \right) $ then (see,$\ $for instance,
\cite{Zol57}) the quantity $\rho $ in (\ref{Spit}) is calculated by the
formula
\begin{equation}
\displaystyle\rho =\left\{
\begin{array}{ll}
\frac{1}{2},\ \text{if \ }\alpha =1, &  \\
\frac{1}{2}+\frac{1}{\pi \alpha }\arctan \left( \beta \tan \frac{\pi \alpha
}{2}\right) ,\text{ otherwise}. &
\end{array}%
\right.   \label{ro}
\end{equation}%
Introduce the following basic assumption:

\textbf{Condition} $A:$ random variables $\left\{ X_{n}=\log f_{n-1}^{\prime
}(1),n\geq 1\right\} $ are independent copies of $X\in \mathcal{D}\left(
\alpha ,\beta \right) $ with $\alpha <2$ and $\left\vert \beta \right\vert
<1 $.

Now we formulate our first result.

\begin{theorem}
\label{T1} Assume that the offspring generating functions are geometric,
i.e.,
\begin{equation}
f_{n-1}(s):=\frac{e^{-X_{n}}}{1+e^{-X_{n}}-s},\,n=1,2,\ldots  \label{geom}
\end{equation}%
with $\{X_{n},n\geq 1\}$ satisfying Condition $A$. Then
\begin{equation}
\mathbf{P}(T=n)\sim \theta \mathbf{P}(T^{-}=n)\sim \theta (1-\rho )\frac{l(n)%
}{n^{2-\rho }}\quad \text{as }n\rightarrow \infty .  \label{T1.1}
\end{equation}
\end{theorem}

\begin{remark}
In the case of geometric offspring distributions one has an explicit formula
for the conditional probability of the event $\{T=n\}$ given the environment
$f_{0},f_{1},\ldots ,f_{n-1}$ in terms of an exponential functional of the
associated random walk $\{S_{k},k\geq 0\}$, see (\ref{condPr}) below. Thus,
the analysis of the extinction probability $\mathbf{P}\left( T=n\right) $ in
this case is reduced  to the study of the expectation of a certain
functional of the associated random walk.
\end{remark}

We now turn to the joint distribution of $T$ and the size $Z_{T-1}$. Here we
don't restrict ourselves to the case of geometric reproduction laws. To
formulate our result we set
\begin{equation*}
\zeta (b):=e^{-2X_{1}}\sum_{k=b}^{\infty }k^{2}f_{0k},\,b=0,1,\ldots
\end{equation*}%
and let $\Lambda :=\left\{ \Lambda _{t},\,0\leq t\leq 1\right\} $ denote the
meander of a strictly stable process with parameters $\alpha ,\beta $, i.e.,
a strictly stable Levy process conditioned to stay positive on the time
interval $(0,1]$ (see \cite{Don85} and \cite{Dur78} for details). Along with
the meander $\Lambda $ consider a stochastic process $\tilde{\Lambda}%
:=\left\{ \tilde{\Lambda}_{t},\,0\leq t\leq 1\right\} $ defined by
\begin{equation*}
\mathbf{E}\left[ \phi \left( \tilde{\Lambda}\right) \right] =\frac{\mathbf{E}%
\left[ \Lambda _{1}^{-\alpha }\phi \left( \Lambda \right) \right] }{\mathbf{E%
}\left[ \Lambda _{1}^{-\alpha }\right] }\text{ for any }\phi \in D\left[ 0,1%
\right] ,
\end{equation*}%
where $D\left[ 0,1\right] $ denotes the space of c\`{a}dl\`{a}g functions on
the unit interval.

\begin{theorem}
\label{T3} Assume that Condition $A$ is valid and there exists $\delta >0$
such that%
\begin{equation*}
\mathbf{E}\left( \log ^{+}\zeta (b)\right) ^{\alpha +\delta }<\infty
\end{equation*}%
for some $b\geq 0$. Then, for every $x>0$,
\begin{equation}
\lim_{n\rightarrow \infty }\frac{\mathbf{P}(Z_{n-1}>e^{xc_{n}};T=n)}{\mathbf{%
P}(T^{-}=n)}=\theta \mathbf{P}\left( \tilde{\Lambda}_{1}>x\right) .
\label{T3.2}
\end{equation}
\end{theorem}

\begin{remark}
It is easy to see that $\zeta (2)\leq 4$ for the geometric offspring
distributions. Therefore, the statement of Theorem~\ref{T3} holds in this
case. Moreover, in view of (\ref{T1.1}),
\begin{equation}
\lim_{x\downarrow 0}\lim_{n\rightarrow \infty }\frac{\mathbf{P}(Z_{n-1}\leq
e^{xc_{n}};T=n)}{\mathbf{P}(T=n)}=0  \label{ttt}
\end{equation}%
provided that the conditions of Theorem \ref{T1} hold.
\end{remark}

We now complement Theorem \ref{T3} by the following statement being valid
for the geometric offspring distributions.

\begin{theorem}
\bigskip \label{T5}Under the conditions of Theorem \ref{T1}, as $%
n\rightarrow \infty$,
\begin{equation*}
\mathcal{L}\left( \frac{\log Z_{\left[ (n-1)t\right] }}{c_{n}},0\leq t\leq
1\Big|\, T=n\right) \Longrightarrow \mathcal{L}\left( \tilde{%
\Lambda}_{t},0\leq t\leq 1\right) .
\end{equation*}
\end{theorem}

Here $\Longrightarrow $ denotes the weak convergence with respect
to the Skorokhod topology in the space $D\left[ 0,1\right] .$

Combining Theorems \ref{T1}, \ref{T3}, and \ref{T5} shows, in particular,
that

\begin{equation}
\lim_{n\rightarrow \infty }\mathbf{P}(Z_{n-1}>e^{xc_{n}}|T=n)=\mathbf{P}%
\left( \tilde{\Lambda}_{1}>x\right) .  \label{T4}
\end{equation}%
in the case when the offspring distributions are geometric. The
last equality, along with Theorem \ref{T5}, allows us to make the
following nonrigorous description of the evolution of a critical
branching process $Z,$ being subject to the conditions of Theorem
\ref{T1}. If the process survives for a long time $(T=n\rightarrow
\infty )$ then $\log Z_{\left[ \left( n-1\right) t\right] }$
grows, roughly speaking, as $c_{n}\tilde{\Lambda}_{t}$
up to moment $n-1$ and then the process instantly extinct. In particular, $%
\log $ $Z_{n-1}$ is of order $c_{n}$ (compare with Corollary 1.6 in \cite%
{AGKV05}). This may be interpreted as the development of the process in a
favorable environment up to the moment $n-1$ and the sudden extinction of
the population at moment $T=n\rightarrow \infty $ because of a very
unfavorable, even "catastrophic" environment at moment $n-1$. At the end of
the paper we show that this phenomenon is in a sharp contrast with the case $%
\mathbf{E}X_{n}=0$, $\sigma ^{2}:=VarX_{n}\in (0,\infty )$. Namely, if,
additionally,
\begin{equation}
\mathbf{E}\left[ \left( 1-f_{00}\right) ^{-1}\right] <\infty ,\quad \mathbf{E%
}\left[ f_{00}^{-1}\right] <\infty ,  \label{add}
\end{equation}%
then
\begin{equation}
\lim_{N\rightarrow \infty }\limsup_{n\rightarrow \infty }\mathbf{P}%
(Z_{n-1}>N|T=n)=0,  \label{eee}
\end{equation}%
while (see Corollary 1.6 in \cite{AGKV05})
\begin{equation*}
\mathcal{L}\left( \frac{\log Z_{\left[ \left( n-1\right) t\right]
}}{\sigma \sqrt{n}},0\leq t\leq 1\Big|\,Z_{n-1}>0\right)
\Longrightarrow \mathcal{L}\left( W_{t}^{+},0\leq t\leq 1\right)
\end{equation*}%
where $W^{+}:=\left\{ W_{t}^{+},0\leq t\leq 1\right\} $ is the Brownian
meander.

These facts demonstrate that the phenomenon of "sudden extinction" in a
favorable environment is absent for the case $\sigma ^{2}<\infty $.
Moreover, one can say that in this case we observe a "natural" extinction of
the population. Indeed, the extinction occurs at moment $T=n$ because of the
small size of the population in the previous generation rather than under
the pressure of the environment.

In the present paper we deal with the annealed approach. As shown in \cite%
{VK08}, one can not see the phenomenon of "sudden extinction" under the
quenched approach even if the conditions of Theorem \ref{T1} are valid. A
"typical" trajectory of a critical branching process in random environment
under the quenched approach oscillates before the extinction. The process
passes through a number of bottlenecks corresponding to the strictly
descending moments of the associated random walk and dies in a "natural" way
because of the small number of individuals in generation $T-1$. Just as
under the annealed approach for the case $\sigma ^{2}<\infty $ (see \cite%
{VK08} for a more detailed discussion).

Another consequence of Theorem~\ref{T3} is the following lower bound for $%
\mathbf{P}(T=n)$.

\begin{corollary}
\label{Col1}Under the conditions of Theorem \ref{T3},
\begin{equation}
\liminf_{n\rightarrow \infty }\frac{\mathbf{P}(T=n)}{\mathbf{P}(T^{-}=n)}%
=\theta .  \label{T3.1}
\end{equation}
\end{corollary}

We conjecture that the relation $\mathbf{P}(T=n)\sim \theta \mathbf{P}%
(T^{-}=n)$ is valid for any critical branching processes in random
environment meeting the conditions of Theorem \ref{T3}, i.e., without the
assumption that the offspring distributions are geometric. With Theorem~\ref%
{T3} in hands, one can easily infer that our conjecture is equivalent to the
equality
\begin{equation*}
\lim_{\varepsilon \rightarrow 0}\lim_{n\rightarrow \infty }\frac{\mathbf{P}%
(Z_{n-1}\leq e^{\varepsilon c_{n}};T=n)}{\mathbf{P}(T^{-}=n)}=0.
\end{equation*}%
But this is exactly the phenomenon of "sudden extinction" described above.
\vspace{12pt}

It is known that there is a natural correspondence between the critical
(subcritical) branching processes in a random environment and the simple
random walks in a random environment with zero (negative) drift. In
particular, (\ref{T1.1}) admits an interpretation in terms of the following
simple random walk $\left\{ R_{k},k\geq 0\right\} $ in random environment.
The walk starts at point $R_{0}:=0$ and has transition probabilities
\begin{align}
& q_{n}:=\mathbf{P}\left( R_{k+1}=n-1|R_{k}=n\right) =\frac{e^{-X_{n+1}}}{%
1+e^{-X_{n+1}}},  \label{DeQ} \\
& p_{n}:=\mathbf{P}\left( R_{k+1}=n+1|R_{k}=n\right) =\frac{1}{1+e^{-X_{n+1}}%
},  \label{DeP}
\end{align}%
$n\in \mathbb{Z}$, where $\{X_{n},\,n\in \mathbb{Z}\}$ are i.i.d. random
variables. Let
\begin{equation*}
\chi :=\min \left\{ k>0:R_{k}=-1\right\} \text{ }
\end{equation*}%
and let%
\begin{equation*}
\ell (n):=\sum_{0\leq k\leq \chi }\mathit{1}\left\{ R_{k}=n\right\} ,\,n\geq
-1,
\end{equation*}%
be the local time of the random walk in random environment
calculated for the first nonnegative excursion. Clearly, if
\begin{equation*}
Z_{n}:=\sum_{i=0}^{n}\left( -1\right) ^{i}\ell (n-i-1),n\geq 0,
\end{equation*}%
then%
\begin{equation*}
\ell (n)=Z_{n+1}+Z_{n},n\geq 0.
\end{equation*}%
One can show that $\left\{ Z_{n},n\geq 0\right\} $ is a branching process in
random environment specified by the offspring generating functions
\begin{equation*}
f_{n}(s):=\frac{q_{n}}{1-p_{n}s}
\end{equation*}%
(see, \cite{VD06} for more detail). In particular, $T:=\min \left\{ j>0:\ell
(j)=0\right\} $ is the extinction moment of the branching process. Clearly,
if $\bar{R}:=\max_{0\leq k<\chi }R_{k}$ then
\begin{equation*}
\text{\ \ }\left\{ \bar{R}=n-1\right\} =\left\{ T=n\right\} .
\end{equation*}%
In these terms Theorem \ref{T1} and relation (\ref{T4})\ are equivalent to
the following statement.

\begin{theorem}
If $q_{n}$ and $p_{n},$specified by (\ref{DeQ})\ and (\ref{DeP}) are such
that
\begin{equation*}
X_{n}:=\log \left( p_{n}/q_{n}\right) ,\ n\in \mathbb{Z},
\end{equation*}%
satisfy Condition $A,$ then, as $n\rightarrow \infty $,
\begin{equation*}
\mathbf{P}\left( \bar{R}=n\right) \sim \theta \mathbf{P}(T^{-}=n).
\end{equation*}%
In addition,
\begin{equation*}
\mathbf{P}(\ell (n)>e^{xc_{n}}|\ell (n)>0,\ell (n+1)=0)\sim \mathbf{P}\left(
\tilde{\Lambda}_{1}>x\right) ,\ x>0,
\end{equation*}%
and, moreover,
\begin{equation*}
\mathcal{L}\left( \frac{\log \ell (\left[ nt\right]
)}{c_{n}},\,0\leq t\leq 1\Big| \,\ell (n)>0,\ell (n+1)=0 \right)
\Longrightarrow \mathcal{L}\left( \tilde{\Lambda}_{t},\,0\leq
t\leq 1\right) .
\end{equation*}
\end{theorem}

Hence, the random walk in random environment visits the maximal possible
level for the first excursion many times provided the length $\chi $ of the
excursion is big. This is essentially different from the case $\mathbf{E}%
\log \left( p_{n}/q_{n}\right) =0,\mathbf{E}\log ^{2}\left(
p_{n}/q_{n}\right) <\infty ,$ where (compare with (\ref{eee}))$\ $%
\begin{equation*}
\lim_{N\rightarrow \infty }\limsup_{n\rightarrow \infty }\mathbf{P}(\ell
(n)>N|\ell (n)>0,\ell (n+1)=0)=0.
\end{equation*}

\section{Some auxiliary results for random walks}

To this aim let us agree to denote by $C,C_{1},C_{2},...,$ some constants
which may be different from formula to formula.

It is known (see, for instance, \cite[Ch. XVII, \S 5]{FE}) that if $X\in
\mathcal{D}(\alpha ,\beta )$ then the scaling sequence
\begin{equation}
c_{n}:=\min \left\{ x>0:\,\mathbf{P}\left( X>x\right) \leq n^{-1}\right\}
,\,n\geq 1,  \label{Defc}
\end{equation}%
for $S_{n}$ is regularly varying with index $\alpha ^{-1}$, i.e., there
exists a function $l_{1}(n),$ slowly varying at infinity, such that
\begin{equation}
c_{n}=n^{1/\alpha }l_{1}(n).  \label{asyma}
\end{equation}%
Moreover, if $X\in \mathcal{D}\left( \alpha ,\beta \right) $ with $\alpha
\in (0,2),$ then
\begin{equation}
\mathbf{P}\left( \left\vert X\right\vert >x\right) \sim \frac{1}{x^{\alpha
}l_{0}(x)}\quad \text{as }x\rightarrow \infty ,  \label{Tailtwo}
\end{equation}%
where $l_{0}(x)$ is a function slowly varying at infinity and
\begin{equation}
\frac{\mathbf{P}\left( X<-x\right) }{\mathbf{P}\left( \left\vert
X\right\vert >x\right) }\rightarrow q,\quad \frac{\mathbf{P}\left(
X>x\right) }{\mathbf{P}\left( \left\vert X\right\vert >x\right) }\rightarrow
p\quad \text{as }x\rightarrow \infty ,  \label{tailF}
\end{equation}%
with $p+q=1$ and $\beta =p-q$ in (\ref{std}). Besides,
\begin{equation}
\mathbf{P}\left( X<-c_{n}\right) \sim \frac{(2-\alpha )q}{\alpha n}\quad
\text{as \ }n\rightarrow \infty   \label{Asx}
\end{equation}%
by (\ref{Defc}) and (\ref{asyma}).

\subsection{Asymptotic behavior of overshoots and undershoots}

In this subsection we prove some results concerning the asymptotic behavior
of the distributions of overshoots and undershoots. We believe that these
results are of independent interest.

Let
\begin{equation*}
\tau ^{-}:=\min \{k\geq 1:S_{k}\leq 0\}.
\end{equation*}

Durrett \cite{Dur78} has shown that if $X\in \mathcal{D}(\alpha ,\beta )$
then
\begin{equation}
\lim_{n\rightarrow \infty }\mathbf{P}(S_{n}\leq xc_{n}|\tau ^{-}>n)=\mathbf{P%
}(\Lambda _{1}\leq x)\quad \text{for all }x\geq 0.  \label{meander}
\end{equation}%
By minor changes of the proof of (\ref{meander}) given by Durrett in \cite%
{Dur78}, one can demonstrate that
\begin{equation}
\lim_{n\rightarrow \infty }\mathbf{P}(S_{n}\leq xc_{n}|T^{-}>n)=\mathbf{P}%
(\Lambda _{1}\leq x)\quad \text{for all }x\geq 0.  \label{meander1}
\end{equation}

We now establish analogs of (\ref{meander}) and (\ref{meander1}) under the
conditions $\left\{ \tau ^{-}=n\right\} $ and $\left\{ T^{-}=n\right\} $.

\begin{lemma}
\label{LdderHeight} If Condition $A$ is valid, then, for any $u>0$,
\begin{equation*}
\lim_{n\rightarrow \infty }\mathbf{P}(S_{n}\leq -uc_{n}|\tau
^{-}=n)=\lim_{n\rightarrow \infty }\mathbf{P}(S_{n}\leq -uc_{n}|T^{-}=n)=%
\frac{\mathbf{E}\left[ (u+\Lambda _{1})^{-\alpha }\right] }{\mathbf{E}\left[
\Lambda _{1}^{-\alpha }\right] }.
\end{equation*}
\end{lemma}

\begin{proof}
For a fixed $u>0$ we have%
\begin{equation}
\mathbf{P}(S_{n}\leq -uc_{n};\tau ^{-}=n)=\mathbf{E}\left[ \mathbf{P}(X\leq
-S_{n-1}-uc_{n});\tau ^{-}>n-1\right] .  \label{3sum}
\end{equation}%
Since, under the conditions of our lemma,
\begin{equation}
\frac{\mathbf{P}(X\leq -x-uc_{n})}{\mathbf{P}(X\leq -c_{n})}%
=(x/c_{n}+u)^{-\alpha }(1+o(1))  \label{Asnu}
\end{equation}%
uniformly in $x\in \lbrack 0,\infty )$, we may approximate for large $n$ the
right-hand side of (\ref{3sum}) by the quantity
\begin{equation*}
\mathbf{P}(X\leq -c_{n})\mathbf{E}\left[ (S_{n-1}/c_{n}+u)^{-\alpha };\tau
^{-}>n-1\right] .
\end{equation*}%
Using (\ref{Asx}) and (\ref{meander}), we obtain
\begin{equation}
\mathbf{P}(S_{n}\leq -uc_{n};\tau ^{-}=n)\sim \frac{q(2-\alpha )}{\alpha n}%
\mathbf{P}(\tau ^{-}>n-1)\mathbf{E}\left[ (\Lambda _{1}+u)^{-\alpha }\right]
.  \label{LH_9}
\end{equation}%
Recall that, by Theorem~7 in \cite{VW07},
\begin{equation}
\mathbf{P}(\tau ^{-}=n)\sim (1-\rho )\frac{\mathbf{P}(\tau ^{-}>n-1)}{n}.
\label{AsimMom}
\end{equation}%
Therefore,
\begin{equation}
\mathbf{P}(S_{n}\leq -uc_{n}|\tau ^{-}=n)\sim \frac{q(2-\alpha )\mathbf{E}%
\left[ (u+\Lambda _{1})^{-\alpha }\right] }{(1-\rho )\alpha }.
\label{local2}
\end{equation}%
This finishes the proof of the first part of the lemma since
\begin{equation}
\mathbf{E}\left[ \Lambda _{1}^{-\alpha }\right] =(1-\rho )\alpha /q(2-\alpha
)  \label{AsE}
\end{equation}%
according to formula (109)\ in \cite{VW07}.

To demonstrate the second part it is sufficient to replace $\tau ^{-}$ by $%
T^{-}$ everywhere in the arguments above.
\end{proof}

\begin{lemma}
\label{LdderJump} If Condition $A$ is valid, then, for any $v>0$,
\begin{eqnarray*}
\lim_{n\rightarrow \infty }\mathbf{P}(S_{n-1}\geq vc_{n}|\tau ^{-}=n)
=\lim_{n\rightarrow \infty }\mathbf{P}(S_{n-1}\geq vc_{n}|T^{-}=n) =\mathbf{P%
}\left( \tilde{\Lambda}_{1}\geq v\right) .
\end{eqnarray*}
\end{lemma}

\begin{proof}
To establish the desired statement one should use the equality%
\begin{eqnarray*}
\mathbf{P}(S_{n-1} &\geq &vc_{n}|\tau ^{-}=n) \\
&=&\frac{\mathbf{P}(\tau ^{-}>n-1)}{\mathbf{P}(\tau ^{-}=n)}\mathbf{E}\left[
\mathbf{P}(X\leq -S_{n-1});S_{n-1}\geq vc_{n}|\tau ^{-}>n-1)\right] ,
\end{eqnarray*}%
and a similar representation with $\tau ^{-}$ replaced by $T^{-},$ the
asymptotic equality (\ref{Asnu}) with $u=0$ and $x\geq vc_{n}$, and the
arguments similar to those applied to demonstrate Lemma \ref{LdderHeight}.
\end{proof}

\begin{remark}
\label{Rem1}It follows from Lemmas \ref{LdderHeight} and \ref{LdderJump}
that the passage from positive to negative (nonnegative) values just at
moment $n$ is possible only owing to a big negative jump of order $c_{n}$ at
this moment. More precisely, in this case both $S_{n-1}$ and $-S_{n}$ are of
order $c_{n}$.
\end{remark}

%%%%%%%%%%%%%%%%%%%%%%%%%%%%%%%%%%%%%%%%%%%%%%%%%%%%%%%%%%%%%%%%%%%%%%%%%%%%%%%%%%%%%%%%%%%%%%%%%%%%

\subsection{Expectations on the event $\{T^-=n\}$}

%%%%%%%%%%%%%%%%%%%%%%%%%%%%%%%%%%%%%%%%%%%%%%%%%%%%%%%%%%%%%%%%%%%%%%%%%%%%%%%%%%%%%
Let
\begin{equation*}
T_{0}:=0,\quad T_{j+1}:=\min (n>T_{j}:S_{n}<S_{T_{j}}),\,j\geq 0,
\end{equation*}%
\ be strictly descending ladder epochs of the random walk $S$. Clearly, $%
T^{-}=T_{1}$. Put $L_{n}:=\min_{0\leq k\leq n}S_{k}$ and introduce the
function
\begin{equation*}
V(x):=\sum_{j=0}^{\infty }\mathbf{P}(S_{T_{j}}\geq -x),\quad x>0,\quad
V\left( 0\right) =1,\quad V\left( x\right) =0,\quad x<0.
\end{equation*}%
The fundamental property of the function $V(x)$ is the identity
\begin{equation}
\mathbf{E}\left[ V(x+X);X+x\geq 0\right] =V(x),\,x\geq 0.  \label{DefV}
\end{equation}

Denote by ${\mathcal{F}}$ the filtration consisting of the $\sigma -$fields $%
{\mathcal{F}}_{n}$ generated by the random variables $S_{0},...,S_{n}$. By
means of $V(x)$ we may specify a probability measure $\mathbf{P}^{+}$ as
follows%
\begin{equation*}
\mathbf{E}^{+}\left[ \psi \left( S_{0},...,S_{n}\right) \right] :=\mathbf{E}%
\left[ \psi \left( S_{0},...,S_{n}\right) V(S_{n});L_{n}\geq 0\right] ,
\end{equation*}%
where $\psi $ is an arbitrary measurable function on the respective space of
arguments. One can check that, in view of (\ref{DefV}), this measure is well
defined (see \cite{AGKV05} for more details).

We now formulate a statement related to the measure $\mathbf{P}^{+}$ which
is a particular case of Lemma 2.5 in \cite{AGKV05}.

\begin{lemma}
\label{Lappbasic3}\cite{AGKV05} Let condition (\ref{Spit}) hold and let $\xi
_{k}$ be a bounded $\mathcal{F}_{k}$-measurable random variable. Then
\begin{equation*}
\lim_{n\rightarrow \infty }\mathbf{E}[\xi _{k}\,|T^{-}>n]=\mathbf{E}^{+}\xi
_{k}.
\end{equation*}%
More general, let $\xi _{1},\xi _{2},\ldots $ be a sequence of uniformly
bounded random variables adopted to the filtration ${\mathcal{F}}$ such that
\begin{equation}
\lim_{n\rightarrow \infty }\xi _{n}=:\xi _{\infty }\qquad   \label{conplus}
\end{equation}%
exists $\mathbf{P}^{+}$ - a.s. Then
\begin{equation*}
\lim_{n\rightarrow \infty }\mathbf{E}[\xi _{n}\,|\,T^{-}>n]=\mathbf{E}%
^{+}\xi _{\infty }.
\end{equation*}
\end{lemma}

We prove a "local" version of this lemma under the additional assumption $%
X\in D(\alpha ,\beta )$. To this aim let
$\tilde{S}:=\{\tilde{S}_{n},n\geq 0\}$ be a probabilistic copy of
$\{S_{n},~n\geq~0\}$. Later on all variables
and expectations related with $\tilde{S}$ \ are supplied with the symbol~$%
^{\sim }$. For instance, we set $\tilde{L}_{n}:=\min_{0\leq k\leq n}%
\tilde{S}_{k}.$

\begin{lemma}
\label{LappbasicLocal}Let $X\in D(\alpha ,\beta )$ with $\alpha <2$ and $%
\beta <1$, and let $\xi _{k}$ be a bounded $\mathcal{F}_{k}$-measurable
random variable. Then
\begin{equation*}
\lim_{n\rightarrow \infty }\mathbf{E}[\xi _{k}\,|\,T^{-}=n]=\mathbf{E}%
^{+}\xi _{k}.
\end{equation*}%
More general, let $\xi _{1},\xi _{2},\ldots $ be a sequence of uniformly
bounded random variables adopted to the filtration ${\mathcal{F}}$ such that
the limit
\begin{equation}
\lim_{n\rightarrow \infty }\xi _{n}=:\xi _{\infty }\qquad   \label{AS}
\end{equation}%
exists $\mathbf{P}^{+}$ - a.s. Then
\begin{equation}
\lim_{n\rightarrow \infty }\mathbf{E}[\xi _{n-1}\,|\,T^{-}=n]=\mathbf{E}%
^{+}\xi _{\infty }.  \label{ASSS}
\end{equation}%
Moreover,%
\begin{equation}
\lim_{n\rightarrow \infty }\mathcal{L}\left( \xi _{n-1}\,|\,T^{-}=n\right) =%
\mathcal{L}\left( \xi _{\infty }\right) .  \label{FFin}
\end{equation}
\end{lemma}

\begin{proof}
According to Lemma \ref{LdderJump}, for any fixed $\varepsilon \in (0,1)$,
\begin{eqnarray}
&&\limsup_{n\rightarrow \infty }\left\vert \mathbf{E}\left[ \xi _{k}\mathrm{1%
}\left\{ \frac{S_{n-1}}{c_{n}}\notin \left[ \varepsilon ,\varepsilon ^{-1}%
\right] \right\} \Big|T^{-}=n\right] \right\vert   \notag \\
&&\qquad \qquad \leq C\limsup_{n\rightarrow \infty }\mathbf{P}\left( \frac{%
S_{n-1}}{c_{n}}\notin \left[ \varepsilon ,\varepsilon ^{-1}\right] \Big|%
T^{-}=n\right)   \notag \\
&&\qquad \qquad \qquad \qquad \qquad \qquad \qquad \leq C\mathbf{P}\left(
\Lambda _{1}\notin \left[ \varepsilon ,\varepsilon ^{-1}\right] \right)
\label{Est10}
\end{eqnarray}%
and this, by Remark \ref{Rem1}, tends to zero as $\varepsilon \downarrow 0$.
Further, \
\begin{eqnarray*}
&&\mathbf{E}\left[ \xi _{k}\mathrm{1}\left\{ \frac{S_{n-1}}{c_{n}}\in \left(
\varepsilon ,\varepsilon ^{-1}\right) \right\} ;T^{-}=n\right]  \\
&&\qquad \qquad =\mathbf{E}\left[ \xi _{k}\mathbf{P}\left( X<-S_{n-1}\right)
\mathrm{1}\left\{ \frac{S_{n-1}}{c_{n}}\in \left[ \varepsilon ,\varepsilon
^{-1}\right] \right\} ;T^{-}>n-1\right] .
\end{eqnarray*}%
Set $\psi _{\varepsilon }(x):=x^{-\alpha }\mathrm{1}(\varepsilon \leq x\leq
\varepsilon ^{-1})$. Since
\begin{equation*}
\mathbf{P}(X<-uc_{n})\sim u^{-\alpha }\mathbf{P}(X<-c_{n})\sim u^{-\alpha }%
\frac{q(2-\alpha )}{\alpha n}
\end{equation*}%
uniformly in $u\in \lbrack \varepsilon ,\varepsilon ^{-1}]$, we have
\begin{align}
& \mathbf{E}\left[ \xi _{k}\mathbf{P}(X<-S_{n-1})\mathrm{1}\left\{ \frac{%
S_{n-1}}{c_{n}}\in \left[ \varepsilon ,\varepsilon ^{-1}\right] \right\}
;T^{-}>n-1\right]   \notag \\
& \sim \mathbf{P}(X<-c_{n})\mathbf{E}\left[ \xi _{k}\left( \frac{S_{n-1}}{%
c_{n}}\right) ^{-\alpha }\mathrm{1}\left\{ \frac{S_{n-1}}{c_{n}}\in \left[
\varepsilon ,\varepsilon ^{-1}\right] \right\} ;T^{-}>n-1\right]   \notag \\
& \qquad \qquad =\frac{q(2-\alpha )}{\alpha n}\mathbf{E}\left[ \xi _{k}\psi
_{\varepsilon }\left( \frac{S_{n-1}}{c_{n}}\right) ;T^{-}>n-1\right] .
\label{Est11}
\end{align}%
\ Conditioning on $S_{0},S_{1},\ldots ,S_{k-1}$ gives
\begin{align*}
& \mathbf{E}\left[ \xi _{k}\psi _{\varepsilon }\left( \frac{S_{n-1}}{c_{n}}%
\right) ;T^{-}>n-1\right] \hspace{4cm} \\
& =\mathbf{E}\left[ \xi _{k}\tilde{\mathbf{E}}\left[ \psi _{\varepsilon
}\left( \frac{\tilde{S}_{n-k}}{c_{n}}\right) ;\tilde{L}_{n-k}\geq -S_{k-1}%
\right] ;T^{-}>k-1\right] .
\end{align*}%
Using Lemmas~2.1 and 2.3 from \cite{AGKV05}, one can easily verify that
\begin{align*}
& \mathbf{E}\left[ \xi _{k}\tilde{\mathbf{E}}\left[ \psi _{\varepsilon
}\left( \frac{\tilde{S}_{n-k}}{c_{n}}\right) ;\tilde{L}_{n-k}\geq -S_{k-1}%
\right] ;T^{-}>k-1\right]  \\
& \hspace{1cm}\sim \mathbf{E}\left[ \xi _{k}\mathbf{P}(\tilde{L}_{n-k}\geq
-S_{k-1});T^{-}>k-1\right] \mathbf{E}[\psi _{\varepsilon }(\Lambda _{1})] \\
& \hspace{2cm}\sim \mathbf{E}\left[ \xi _{k}V(S_{k-1});T^{-}>k-1\right]
\mathbf{P}(T^{-}>n-k)\mathbf{E}[\psi _{\varepsilon }(\Lambda _{1})] \\
& \qquad \qquad \qquad \qquad \sim \mathbf{E}^{+}\left[ \xi _{k}\right]
\mathbf{P}(T^{-}>n-1)\mathbf{E}[\psi _{\varepsilon }(\Lambda _{1})].
\end{align*}%
Thus,
\begin{eqnarray}
&&\mathbf{E}\left[ \xi _{k}\psi _{\varepsilon }\left( \frac{S_{n-1}}{c_{n}}%
\right) \mathrm{1}\left\{ \frac{S_{n-1}}{c_{n}}\in \left[ \varepsilon
,\varepsilon ^{-1}\right] \right\} \Big|T^{-}=n\right]   \notag \\
&&\qquad \qquad \sim \frac{\mathbf{P}(T^{-}>n-1)}{\mathbf{P}(T^{-}=n)}\frac{%
q(2-\alpha )}{\alpha n}\mathbf{E}^{+}\left[ \xi _{k}\right] \mathbf{E}[\psi
_{\varepsilon }(\Lambda _{1})].  \label{Est133}
\end{eqnarray}%
Clearly, $\mathbf{E}[\psi _{\varepsilon }(\Lambda _{1})]\rightarrow \mathbf{E%
}\Lambda _{1}^{-\alpha }$ as $\varepsilon \rightarrow 0.$ Combining these
estimates with (\ref{AsE}), (\ref{tails}) and the asymptotic relation%
\begin{equation}
\mathbf{P}(T^{-}=n)\sim \frac{(1-\rho )}{n}\mathbf{P}(T^{-}>n-1)=(1-\rho )%
\frac{l(n)}{n^{2-\rho }},  \label{**}
\end{equation}%
established in Theorem 8 of \cite{VW07}, and recalling (\ref{Est10}), we
complete the proof of the first part of the lemma.

To show the second part we fix an $\varepsilon \in (0,1)$ and write%
\begin{eqnarray*}
\left\vert \mathbf{E}\left[ \xi _{k}-\xi _{n-1}|T^{-}=n\right] \right\vert
&\leq &\mathbf{E}\left[ \left\vert \xi _{k}-\xi _{n-1}\right\vert \mathrm{1}%
\left\{ \frac{S_{n-1}}{c_{n}}\notin \left[ \varepsilon ,\varepsilon ^{-1}%
\right] \right\} \Big|T^{-}=n\right]  \\
&&+\mathbf{E}\left[ \left\vert \xi _{k}-\xi _{n-1}\right\vert \mathrm{1}%
\left\{ \frac{S_{n-1}}{c_{n}}\in \left[ \varepsilon ,\varepsilon ^{-1}\right]
\right\} \Big|T^{-}=n\right] .
\end{eqnarray*}%
Similarly to (\ref{Est10}),
\begin{align}
& \lim_{\varepsilon \downarrow 0}\limsup_{n\rightarrow \infty }\mathbf{E}%
\left[ \left\vert \xi _{k}-\xi _{n-1}\right\vert \mathrm{1}\left\{ \frac{%
S_{n-1}}{c_{n}}\notin \left[ \varepsilon ,\varepsilon ^{-1}\right] \right\} %
\Big|T^{-}=n\right]   \notag \\
& \qquad \qquad \leq C\lim_{\varepsilon \downarrow 0}\limsup_{n\rightarrow
\infty }\mathbf{P}\left( \frac{S_{n-1}}{c_{n}}\notin \left[ \varepsilon
,\varepsilon ^{-1}\right] \Big|T^{-}=n\right) =0,  \label{time1}
\end{align}%
while, by analogy with (\ref{Est11}) and (\ref{Est133}),
\begin{align}
& \mathbf{E}\left[ \left\vert \xi _{k}-\xi _{n-1}\right\vert \mathrm{1}%
\left\{ \frac{S_{n-1}}{c_{n}}\in \left[ \varepsilon ,\varepsilon ^{-1}\right]
\right\} \Big|T^{-}=n\right]   \notag \\
& \qquad \qquad \leq C\mathbf{E}\left[ \left\vert \xi _{k}-\xi
_{n-1}\right\vert \psi _{\varepsilon }\left( \frac{S_{n-1}}{c_{n}}\right) %
\Big|T^{-}>n-1\right]   \notag \\
& \qquad \qquad \qquad \leq C\varepsilon ^{-\alpha }\mathbf{E}\left[
\left\vert \xi _{k}-\xi _{n-1}\right\vert \Big|T^{-}>n-1\right] .
\label{time2}
\end{align}%
We know by Lemma \ref{Lappbasic3} that, given (\ref{AS}),
\begin{equation}
\lim_{k\rightarrow \infty }\lim_{n\rightarrow \infty }\mathbf{E}\left[
\left\vert \xi _{k}-\xi _{n-1}\right\vert |T^{-}>n-1\right]
=\lim_{k\rightarrow \infty }\mathbf{E}^{+}\left\vert \xi _{k}-\xi _{\infty
}\right\vert =0.  \label{time3}
\end{equation}%
Combining (\ref{time1})-(\ref{time3}) completes the proof of the second part
of the lemma.

To prove (\ref{FFin}) it is sufficient to observe that, by (\ref{ASSS}) and
the dominated convergence theorem,
\begin{equation*}
\lim_{n\rightarrow \infty }\mathbf{E}[e^{it\xi _{n-1}}\,|\,T^{-}=n]=\mathbf{E%
}^{+}\left[ e^{it\xi _{\infty }}\right],\,t\in (-\infty ,\infty ).
\end{equation*}
\end{proof}

Set $\mu _{n}:=\min \{k\geq 0:S_{k}=L_{n}\}$.

\begin{lemma}
\label{LMinMax}If $X\in \mathcal{D}(\alpha ,\beta )$ then%
\begin{equation*}
\limsup_{n\rightarrow \infty }nc_{n}\mathbf{E}\left[ e^{2L_{n}-S_{n}}\right]
<\infty
\end{equation*}%
and%
\begin{equation}
\limsup_{n\rightarrow \infty }nc_{n}\mathbf{E}\left[ e^{S_{n}};\mu _{n}=n%
\right] <\infty .  \label{TT}
\end{equation}
\end{lemma}

\begin{proof}
By the factorization identity (see, for instance, Theorem 8.9.3 in \cite{BGT}%
) applied with $\lambda =-1$ and $\mu =1$ to $L_{n}$ instead of $%
M_{n}:=\max_{0\leq k\leq n}S_{k},$ we have%
\begin{equation*}
\sum_{n=0}^{\infty }r^{n}\mathbf{E}\left[ e^{2L_{n}-S_{n}}\right] =\exp
\left\{ \sum_{n=1}^{\infty }\frac{r^{n}}{n}\left( \mathbf{E}\left[
e^{-S_{n}};S_{n}\geq 0\right] +\mathbf{E}\left[ e^{S_{n}};S_{n}<0\right]
\right) \right\} .
\end{equation*}%
Since $X\in D(\alpha ,\beta ),$ the local limit theorem for asymptotically
stable random walks implies
\begin{equation}
\mathbf{E}\left[ e^{-S_{n}};S_{n}\geq 0\right] +\mathbf{E}\left[
e^{S_{n}};S_{n}<0\right] \leq \frac{C}{c_{n}}.  \label{UU}
\end{equation}%
Combining this with Theorem 6 in \cite{Rog76} gives%
\begin{equation*}
\limsup_{n\rightarrow \infty }nc_{n}\mathbf{E}\left[ e^{2L_{n}-S_{n}}\right]
<\infty ,
\end{equation*}%
proving the first statement of the lemma.

To prove the second it is sufficient to\ note (see, for instance, Theorem
8.9.1 in~\cite{BGT}) that
\begin{equation*}
\sum_{n=0}^{\infty }r^{n}\mathbf{E}\left[ e^{S_{n}};\mu _{n}=n\right] =\exp
\left\{ \sum_{n=1}^{\infty }\frac{r^{n}}{n}\mathbf{E}\left[ e^{S_{n}};S_{n}<0%
\right] \right\}
\end{equation*}%
and to use estimate (\ref{UU}) once again.
\end{proof}

The previous lemma allows us to prove the following statement.

\begin{lemma}
\label{LtailSer}If $X\in \mathcal{D}(\alpha ,\beta )$ with $\alpha<2$ and $%
\beta<1$, then for every $\varepsilon >0$ there exists a positive integer $l$
such that
\begin{equation*}
\sum_{k=l}^{n-1}\mathbf{E}\left[ e^{S_{k}};\mu _{k}=k\right] \mathbf{P}%
\left( T^{-}=n-k\right) \leq \varepsilon \mathbf{P}\left( T^{-}=n\right)
\end{equation*}
for all $n\geq l$.
\end{lemma}

\begin{proof}
By Lemma \ref{LMinMax} and (\ref{**}) we have, for any $\delta \in (0,1)$,
\begin{eqnarray}
&&\sum_{k=l}^{n}\mathbf{E}\left[ e^{S_{k}};\mu _{k}=k\right] \mathbf{P}%
\left( T^{-}=n-k\right) \hspace{4cm}  \notag \\
&\leq &\max_{n\delta \leq j\leq n}\mathbf{P}\left( T^{-}=j\right)
\sum_{l\leq k\leq (1-\delta )}\mathbf{E}\left[ e^{S_{k}};\mu _{k}=k\right]
\notag \\
&&+\frac{C}{n(1-\delta )c_{n(1-\delta )}}\sum_{k\leq n\delta }\mathbf{P}%
\left( T^{-}=n-k\right) .  \label{H0}
\end{eqnarray}%
On account of (\ref{**}),
\begin{equation}
\max_{n\delta \leq j\leq n}\mathbf{P}\left( T^{-}=j\right) \leq C\delta
^{\rho -2}\frac{l(n)}{n^{2-\rho }}\leq C_{1}\delta ^{\rho -2}\mathbf{P}%
\left( T^{-}=n\right) .  \label{H1}
\end{equation}%
Using (\ref{ro}) it is not difficult to check that $1-\rho <\alpha ^{-1}$ if
Condition $A$ holds. With this in view we have, by (\ref{asyma}) and (\ref%
{**}),
\begin{equation}
\frac{1}{nc_{n(1-\delta )}}\leq \frac{C}{(1-\delta )^{1/\alpha
}n^{1+1/\alpha }l_{1}(n)}\leq \frac{C_{1}}{(1-\delta )^{1/\alpha
}n^{1/\alpha +\rho -1}}\mathbf{P}\left( T^{-}=n\right) .  \label{H2}
\end{equation}%
Substituting (\ref{H1}) and (\ref{H2}) in (\ref{H0}) gives%
\begin{eqnarray*}
&&\sum_{k=l}^{n}\mathbf{E}\left[ e^{S_{k}};\mu _{k}=k\right] \mathbf{P}%
\left( T^{-}=n-k\right)  \\
&&\qquad \leq C_{3}\mathbf{P}\left( T^{-}=n\right) \left( \delta ^{\rho
-2}\sum_{k=l}^{\infty }\mathbf{E}\left[ e^{S_{k}};\mu _{k}=k\right] +\frac{1%
}{(1-\delta )^{1+1/\alpha }n^{1/\alpha +\rho -1}}\right)
\end{eqnarray*}%
for sufficiently large $C_{3}.$ Recalling now (\ref{TT}), we complete the
proof of the lemma by an appropriate choice of $\delta $ and $l$.
\end{proof}

%%%%%%%%%%%%%%%%%%%%%%%%%%%%%%%%%%%%%%%%%%%%%%%%%%%%%%%%%%%%%%%%%%%%%%%%%%%%%%%%%%%%%%%%%

\section{Proof of Theorem~\protect\ref{T1}}

%%%%%%%%%%%%%%%%%%%%%%%%%%%%%%%%%%%%%%%%%%%%%%%%%%%%%%%%%%%%%%%%%%%%%%%%%%%%%%%%%%%%%%%%%
Set%
\begin{equation*}
F_{m,n}(s):=f_{m}(f_{1}(\ldots (f_{n-1}(0))\ldots )),\,m<n,\quad
F_{n,n}(s):=s.
\end{equation*}%
Rewriting (\ref{geom}) as
\begin{equation*}
\frac{1}{1-f_{n-1}(s)}=1+e^{-X_{n}}\frac{1}{1-s}\quad \text{for all }n\geq 1,
\end{equation*}%
one can easily get the representation
\begin{equation}
\frac{1}{1-F_{0,n}(s)}=1+e^{-S_{1}}+e^{-S_{2}}+\ldots
+e^{-S_{n-1}}+e^{-S_{n}}\frac{1}{1-s}  \label{ddd}
\end{equation}%
for all $n\geq 1$. From this equality, setting
\begin{equation*}
H_{n}:=\left( \sum_{k=0}^{n}e^{-S_{k}}\right) ^{-1},\,\quad H_{\infty
}:=\lim_{n\rightarrow \infty }H_{n},
\end{equation*}%
we get
\begin{equation*}
\mathbf{P}_{f}(Z_{n}>0):=\mathbf{P}(Z_{n}>0|f_0,f_1,%
\ldots,f_{n-1})=1-F_{0,n}(0)=H_{n}
\end{equation*}%
and
\begin{equation}  \label{condPr}
\mathbf{P}_{f}(T=n):=\mathbf{P}_{f}(Z_{n-1}>0)-\mathbf{P}%
_{f}(Z_{n}>0)=H_{n-1}H_{n}e^{-S_{n}}.
\end{equation}

We split the expectation $\mathbf{E}\left[ \mathbf{P}_{f}(T=n)\right] =%
\mathbf{P}\left( T=n\right) $ into two parts:
\begin{equation}
\mathbf{P}\left( T=n\right) =\mathbf{E}[\mathbf{P}_{f}(T=n);\mu _{n}<n]+%
\mathbf{E}[\mathbf{P}_{f}(T=n);\mu _{n}=n].  \label{1}
\end{equation}%
One can easily verify that $H_{n-1}H_{n}e^{-S_{n}}\leq e^{2L_{n}-S_{n}}$ on
the event $\{\mu _{n}<n\}$. From this bound and Lemma \ref{LMinMax} we infer
\begin{equation}
\mathbf{E}[\mathbf{P}_{f}\left( T=n\right) ;\mu _{n}<n]\leq \mathbf{E}%
e^{2L_{n}-S_{n}}\leq \frac{C}{nc_{n}}\quad \text{for all }n\geq 1.  \label{2}
\end{equation}%
Using estimate (\ref{H2}) with $\delta =0$, we conclude
\begin{equation}
\mathbf{E}[\mathbf{P}_{f}\left( T=n\right) ;\mu _{n}<n]=o\left( \mathbf{P}%
\left( T^{-}=n\right) \right) .  \label{SSS}
\end{equation}%
Consider now the expectation $\mathbf{E}[\mathbf{P}_{f}\left( T=n\right)
;\mu _{n}=n]$. Applying Lemma \ref{LMinMax} once again, we see that
\begin{eqnarray*}
\mathbf{E}\left[ \mathbf{P}_{f}(Z_{n}>0);\mu _{n}=n\}\right]  &\leq &\mathbf{%
E}\left[ e^{S_{n}};\mu _{n}=n\}\right]  \\
&\leq &\frac{C}{nc_{n}}=o\left( \mathbf{P}\left( T^{-}=n\right) \right) .
\end{eqnarray*}%
Thus,
\begin{equation}
\mathbf{E}[\mathbf{P}_{f}\left( T=n\right) ;\mu _{n}=n\}]=\mathbf{E}[\mathbf{%
P}_{f}(Z_{n-1}>0);\mu _{n}=n\}]+o(\mathbf{P}(T^{-}=n)).  \label{9}
\end{equation}%
Since $\mathbf{P}_{f}(Z_{n-1}>0)\leq e^{\min_{0\leq j\leq n-1}S_{j}},$ we
have by Lemma \ref{LtailSer} that for any $\varepsilon >0$ there exists $l$
such that for all $n>l$
\begin{eqnarray}
\mathbf{E}[\mathbf{P}_{f}(Z_{n-1} &>&0);\mu _{n-1}\geq l,\mu _{n}=n\}]
\notag \\
&=&\sum_{k=l}^{n-1}\mathbf{E}[\mathbf{P}_{f}(Z_{n-1}>0);\mu _{n-1}=k,\mu
_{n}=n\}]  \notag \\
&\leq &\sum_{k=l}^{n-1}\mathbf{E}[e^{S_{k}};\mu _{n-1}=k,\mu _{n}=n\}]\leq
\varepsilon \mathbf{P}(T^{-}=n)  \label{RRR}
\end{eqnarray}%
for all $n\geq l$. Denoting by $\left\{ \tilde{f}_{n},n\geq 0\right\} $ a
probabilistic and independent copy of $\left\{ f_{n},n\geq 0\right\} $ we
have, for any fixed $k<l$,
\begin{align*}
\mathbf{E}[\mathbf{P}_{f}& (Z_{n-1}>0);\mu _{n-1}=k,\mu _{n}=n\}] \\
& =\mathbf{E}\left[ \left( 1-F_{0,n-2}(0)\right) ;\mu _{n-1}=k,\mu _{n}=n\}%
\right]  \\
& =\mathbf{E}\left[ \left( 1-F_{0,k}(F_{k,n-2}(0))\right) ;\mu _{n-1}=k,\mu
_{n}=n\}\right]  \\
& =\mathbf{E}\left[ \left( 1-F_{0,k}(\tilde{F}_{0,n-k-2}(0))\right) ;\mu
_{k}=k,\tilde{T}^{-}=n-k\}\right]  \\
& =\mathbf{E}\left[ \left( 1-F_{0,k}(\tilde{F}_{0,n-k-2}(0))\right) \mathit{1%
}\{\mu _{k}=k\}|\tilde{T}^{-}=n-k\right] \mathbf{P}\left( T^{-}=n-k\right) .
\end{align*}%
By monotonicity of the extinction probability and Lemma 2.7 in \cite{AGKV05},%
\begin{equation*}
\lim_{n\rightarrow \infty }\tilde{F}_{0,n}(0)=:Q^{+}<1\quad \mathbf{P}^{+}%
\text{-a.s.}
\end{equation*}%
Hence, in view of (\ref{FFin}) we get for any fixed $k<l:$
\begin{eqnarray*}
&&\mathbf{E}\left[ \left( 1-F_{0,k}(\tilde{F}_{0,n-k-2}(0))\right) \mathit{1}%
\{\mu _{k}=k\}|\tilde{T}^{-}=n-k\right]  \\
&&\qquad \qquad \qquad \sim \mathbf{E}\left[ \mathbf{E}^{+}\left[ \left(
1-F_{0,k}(Q^{+})\right) \right] ;\mu _{k}=k\right] >0.
\end{eqnarray*}%
Using this relation, (\ref{**}) and (\ref{RRR}) it is not difficult to show
that%
\begin{equation}
\mathbf{E}[\mathbf{P}_{f}(Z_{n-1}>0);\mu _{n}=n]\sim \theta \mathbf{P}\left(
T^{-}=n\right) ,  \label{SSemi}
\end{equation}%
where%
\begin{equation}
\theta =\sum_{k=0}^{\infty }\mathbf{E}\left[ \mathbf{E}^{+}\left(
1-F_{0,k}(Q^{+})\right) ;\mu _{k}=k\right] >0.  \label{NUU}
\end{equation}

It follows from (\ref{1}), (\ref{SSS}), (\ref{9}) and (\ref{SSemi}) that%
\begin{equation*}
\mathbf{P}\left( T=n\right) \sim \theta \mathbf{P}\left( T^{-}=n\right) .
\end{equation*}%
It is easy to check that the expression for $\theta $ given by (\ref{NUU})
is in complete agreement with formula (4.10) in \cite{AGKV05}. This finishes
the proof of Theorem \ref{T1}.
%%%%%%%%%%%%%%%%%%%%%%%%%%%%%%%%%%%%%%%%%%%%%%%%%%%%%%%%%%%%%%%%%%%%%%%%%%%%%%%%%%%%%%%%%%%%%
%%%%%%%%%%%%%%%%%%%%%%%%%%%%%%%%%%%%%%%%%%%%%%%%%%%%%%%%%%%%%%%%%%%%%%%%%%%%%%%%%%%%%%%%%%%%%
%%%%%%%%%%%%%%%%%%%%%%%%%%%%%%%%%%%%%%%%%%%%%%%%%%%%%%%%%%%%%%%%%%%%%%%%%%%%%%%%%%%%%%%%%%%%%

\section{Proofs for the general case}

\textit{Proof of Theorem~\ref{T3}}. First we obtain lower and upper bounds
for the probability $\mathbf{P}(Z_{1}=0|Z_{0}=k)$. It is easy to see that
(recall (\ref{DefF}))
\begin{equation*}
f_{00}=\mathbf{P}(Z_{1}=0|Z_{0}=k;f_{0})\geq \max \left\{
0,1-\sum_{k=1}^{\infty }kf_{0k}\right\} .
\end{equation*}%
Therefore, for any fixed $\varepsilon \in \left( 0,1/2\right) $,
\begin{align}
\mathbf{P}(Z_{1}=0|Z_{0}=k)& \geq \mathbf{E}[(1-e^{X_{1}})^{k};X_{1}<0]
\notag \\
& \geq (1-k^{-1-\varepsilon })^{k}\mathbf{P}(X_{1}\leq -(1+\varepsilon )\log
k).  \label{Llow}
\end{align}%
To get an upper estimate we use the inequality $\mathbf{P}(Y>0)\geq \left(
\mathbf{E}Y\right) ^{2}/\mathbf{E}Y^{2},$ being valid for any nonnegative
random variables with $\mathbf{E}Y>0,$ to conclude that
\begin{equation*}
f_{00}\leq 1-\left( \frac{\sum_{k=1}^{\infty }k^{2}f_{0k}}{\left(
\sum_{k=1}^{\infty }kf_{0k}\right) ^{2}}\right) ^{-1}.
\end{equation*}%
Observing that%
\begin{equation*}
\frac{\sum_{k=1}^{\infty }k^{2}f_{0k}}{\left( \sum_{k=1}^{\infty
}kf_{0k}\right) ^{2}}\leq \frac{b}{\sum_{k=1}^{\infty }kf_{0k}}+\zeta (b),
\end{equation*}%
we get%
\begin{equation*}
f_{00}\leq 1-\left( be^{-X_{1}}+\zeta (b)\right) ^{-1}\leq \exp \left\{ -%
\frac{1}{be^{-X_{1}}+\zeta (b)}\right\} .
\end{equation*}%
This implies
\begin{eqnarray}
\mathbf{P}(Z_{1} &=&0|Z_{0}=k)\leq \mathbf{E}\left[ \exp \left\{ -\frac{1}{%
be^{-X_{1}}+\zeta (b)}\right\} \right]   \notag \\
&\leq &\mathbf{P}(X_{1}\leq -(1-\varepsilon )\log k)+\mathbf{P}\left( \zeta
(b)>k^{1-\varepsilon }\right) +\exp \left\{ -\frac{k^{\varepsilon }}{b+1}%
\right\} .  \label{Upper2}
\end{eqnarray}%
In view of the hypothesis $\mathbf{E}\left( \log ^{+}\zeta (b)\right)
^{\alpha +\delta }<\infty $ and the Markov inequality we have%
\begin{equation}
\mathbf{P}\left( \zeta (b)>k^{1-\varepsilon }\right) \leq C\log ^{-\alpha
-\delta }k.  \label{Upper3}
\end{equation}%
Since the probability $\mathbf{P}\left( X_{1}<-x\right) $ is regularly
varying with index $-\alpha $, estimates (\ref{Llow})- (\ref{Upper3}) imply
\begin{align*}
\frac{1}{\left( 1+\varepsilon \right) ^{\alpha }}& \leq
\liminf_{k\rightarrow \infty }\frac{\mathbf{P}(Z_{1}=0|Z_{0}=k)}{\mathbf{P}%
\left( X_{1}<-\log k\right) } \\
& \quad \leq \limsup_{k\rightarrow \infty }\frac{\mathbf{P}(Z_{1}=0|Z_{0}=k)%
}{\mathbf{P}\left( X_{1}<-\log k\right) }\leq \frac{1}{\left( 1-\varepsilon
\right) ^{\alpha }}.
\end{align*}%
Letting $\varepsilon \rightarrow 0$ gives%
\begin{equation*}
\lim_{k\rightarrow \infty }\frac{\mathbf{P}(Z_{1}=0|Z_{0}=k)}{\mathbf{P}%
\left( X_{1}<-\log k\right) }=1.
\end{equation*}%
Therefore,
\begin{align}
\mathbf{P}(Z_{n-1}>e^{xc_{n}};T=n)& =\sum_{k>e^{xc_{n}}}\mathbf{P}(Z_{n-1}=k)%
\mathbf{P}(Z_{1}=0|Z_{0}=k)  \notag \\
& \sim \sum_{k>e^{xc_{n}}}\mathbf{P}(Z_{n-1}=k)\mathbf{P}(X\leq -\log k)
\notag \\
& \sim \mathbf{E}[\mathbf{P}(X\leq -\log Z_{n-1});\log Z_{n}>xc_{n}].
\label{New0}
\end{align}%
Since
\begin{equation*}
\frac{\mathbf{P}(X\leq -yc_{n})}{\mathbf{P}(X\leq -c_{n})}\rightarrow
y^{-\alpha }\text{ \ as \ }n\rightarrow \infty
\end{equation*}%
uniformly in $y\in \lbrack x,\infty )$, we have
\begin{align}
& \mathbf{E}[\mathbf{P}(X\leq -\log Z_{n-1});\log Z_{n-1}>xc_{n}]  \notag \\
& \sim \mathbf{P}(X\leq -c_{n})\mathbf{E}\left[ \left( \frac{\log Z_{n-1}}{%
c_{n}}\right) ^{-\alpha };\log Z_{n-1}>xc_{n}\right]   \notag \\
& \sim \mathbf{P}(X\leq -c_{n})\mathbf{P}(Z_{n-1}>0)\mathbf{E}\left[ \left(
\frac{\log Z_{n-1}}{c_{n}}\right) ^{-\alpha }\mathit{1}\left\{ \log
Z_{n-1}>xc_{n}\right\} \Big|Z_{n-1}>0\right] .  \label{New1}
\end{align}%
By Corollary~1.6 in \cite{AGKV05},%
\begin{equation*}
\lim_{n\rightarrow \infty }\mathbf{P}\left( \frac{\log Z_{n-1}}{c_{n}}<x\Big|%
Z_{n-1}>0\right) =\mathbf{P}\left( \Lambda _{1}<x\right) ,\ x>0.
\end{equation*}%
This and the dominated convergence theorem yield%
\begin{equation}
\lim_{n\rightarrow \infty }\mathbf{E}\left[ \left( \frac{\log Z_{n-1}}{c_{n}}%
\right) ^{-\alpha }\mathit{1}\left\{ \log Z_{n-1}>xc_{n}\right\} \Big|%
Z_{n-1}>0\right] =\mathbf{E}\left[ \Lambda _{1}^{-\alpha }\mathit{1}\left\{
\Lambda _{1}>x\right\} \right]   \label{New2}
\end{equation}%
Combining (\ref{New0})-(\ref{New2}) and taking into account (\ref{Asx}) and (%
\ref{AsE}), we obtain%
\begin{eqnarray*}
\mathbf{P}(Z_{n-1} &>&e^{xc_{n}};Z_{n}=0)\sim (1-\rho )\frac{\mathbf{P}%
(Z_{n-1}>0)}{n}\frac{\mathbf{E}\left[ \Lambda _{1}^{-\alpha }\mathit{1}%
\left\{ \Lambda _{1}>x\right\} \right] }{\mathbf{E}\left[ \Lambda
_{1}^{-\alpha }\right] } \\
&=&(1-\rho )\frac{\mathbf{P}(Z_{n-1}>0)}{n}\mathbf{P}\left( \tilde{\Lambda}%
_{1}>x\right) .
\end{eqnarray*}%
To complete the proof of Theorem \ref{T3} it remains to note that
\begin{equation*}
(1-\rho )\frac{\mathbf{P}(Z_{n-1}>0)}{n}\sim (1-\rho )\frac{\theta \mathbf{P}%
(T^{-}>n-1)}{n}\sim \theta \mathbf{P}(T^{-}=n)
\end{equation*}%
in view of (\ref{tails}) and (\ref{**}).

\textit{Proof of Theorem \ref{T5}}. Let $\phi $ be an arbitrary bounded
continuous function from $D\left[ 0,1\right] $ and let
\begin{equation*}
Z^{(n)}=\left\{ \frac{\log Z_{\left[ (n-1)t\right] }}{c_{n}},0\leq t\leq
1\right\} .
\end{equation*}%
As in the proof of Theorem~\ref{T3}, for any $x>0$,%
\begin{eqnarray}
&&\left. \sum_{k>e^{xc_{n}}}\mathbf{E}\left[ \phi \left( Z^{(n)}\right)
;Z_{n-1}=k\right] \mathbf{P}(Z_{1}=0|Z_{0}=k)\right.   \notag \\
&\sim &\mathbf{P}(X\leq -c_{n})\mathbf{E}\left[ \phi \left( Z^{(n)}\right)
\left( \frac{\log Z_{n-1}}{c_{n}}\right) ^{-\alpha }\mathit{1}\left\{
Z_{n-1}>e^{xc_{n}}\right\} \right]   \notag \\
&\sim &\theta \mathbf{P}(T^{-}=n)\mathbf{E}\left[ \phi \left( Z^{(n)}\right)
\left( \frac{\log Z_{n-1}}{c_{n}}\right) ^{-\alpha }\mathit{1}\left\{
Z_{n-1}>e^{xc_{n}}\right\} \Big|\,Z_{n-1}>0\right]   \notag \\
&\sim &\mathbf{P}(T=n)\frac{\mathbf{E}\left[ \phi \left( \Lambda \right)
\Lambda _{1}^{-\alpha }1\left\{ \Lambda _{1}>x\right\} \right] }{\mathbf{E}%
\left[ \Lambda _{1}^{-\alpha }\right] } \\
&=&\mathbf{P}(T=n)\mathbf{E}\left[ \phi \left( \tilde{\Lambda}\right)
\mathit{1}\left\{ \tilde{\Lambda}_{1}>x\right\} \right] ,  \label{Func}
\end{eqnarray}%
where in the last step we have used Corollary 1.6 in \cite{AGKV05}.

On the other hand, according to (\ref{ttt}),
\begin{align}
& \left. \sum_{0<k\leq e^{xc_{n}}}\mathbf{E}\left[ \phi \left(
Z^{(n)}\right) ;Z_{n-1}=k\right] \mathbf{P}(Z_{1}=0|Z_{0}=k)\right.  \notag
\\
& \leq \sup \left\vert \phi \right\vert \mathbf{P}\left( 0<Z_{n-1}\leq
e^{xc_{n}};Z_{n}=0\right) =o\left( \mathbf{P}(T=n)\right)  \label{fff}
\end{align}%
as $x\downarrow 0$. Combining (\ref{Func}) and (\ref{fff}), we get%
\begin{equation*}
\lim_{n\rightarrow \infty }\mathbf{E}\left[ \phi \left( Z^{(n)}\right) |T=n%
\right] \mathbf{=}\lim_{x\downarrow 0}\mathbf{E}\left[ \phi \left( \tilde{%
\Lambda}\right) \mathit{1}\left\{ \tilde{\Lambda}_{1}>x\right\} \right] =%
\mathbf{E}\left[ \phi \left( \tilde{\Lambda}\right) \right]
\end{equation*}%
completing the proof of Theorem \ref{T5}.

\textit{Proof of Corollary \ref{Col1}}. Letting $x\rightarrow 0$ in (\ref%
{T3.2}), we get%
\begin{equation}
\liminf_{n\rightarrow \infty }\frac{\mathbf{P}(T=n)}{\mathbf{P}(T^{-}=n)}%
\geq \theta .  \label{Lala}
\end{equation}%
Assuming that there exists $\varepsilon >0$ such that
\begin{equation*}
\mathbf{P}(T=n)\geq \left( \theta +\varepsilon \right) \mathbf{P}(T^{-}=n)
\end{equation*}%
for all $n\geq N$ and summing over $n$ from arbitrary $n_{0}>N$ to $\infty$,
we deduce
\begin{equation*}
\mathbf{P}(T\geq n_{0})\geq \left( \theta +\varepsilon \right) \mathbf{P}%
(T^{-}\geq n_{0})
\end{equation*}%
for all $n_{0}\geq N$, that contradicts (\ref{tails}).

\textit{Proof of }(\ref{eee}). Representation (\ref{ddd}) implies%
\begin{equation*}
\mathbf{P}\left( Z_{n-1}=j\right) =\mathbf{E}\left[ H_{n-1}^{2}e^{-S_{n-1}}%
\left( 1-H_{n-1}e^{-S_{n-1}}\right) ^{j-1}\right] ,\,j\geq 1.
\end{equation*}%
Hence, by Lemma \ref{LMinMax},
\begin{eqnarray*}
\sup_{j\geq 1}\mathbf{P}\left( Z_{n-1}=j\right)  &\leq &\mathbf{E}\left[
H_{n-1}^{2}e^{-S_{n-1}}\right]  \\
&\leq &\mathbf{E}\left[ e^{2L_{n-1}-S_{n-1}}\right] \leq \frac{C}{nc_{n}}%
\leq \frac{C_{1}}{\sigma n^{3/2}},
\end{eqnarray*}%
where in the last step we have used the equality $c_{n}\sim \sigma \sqrt{n}$%
. Now we see that%
\begin{eqnarray}
\mathbf{P}\left( Z_{n-1}>N;T=n\right)  &=&\sum_{j=N+1}^{\infty }\mathbf{P}%
\left( Z_{n-1}=j\right) \mathbf{E}\left[ f_{00}^{j}\right]   \notag \\
&\leq &\frac{C_{1}}{\sigma n^{3/2}}\sum_{j=N+1}^{\infty }\mathbf{E}\left[
f_{00}^{j}\right] .  \label{kkk}
\end{eqnarray}%
According to Theorem 1 in \cite{VD97} the conditions $\sigma ^{2}<\infty $
and (\ref{add}) yield $\mathbf{P}\left( T=n\right) \sim Cn^{-3/2}$. From
this estimate, the first condition in (\ref{add}), and (\ref{kkk}) we get (%
\ref{eee}).

\end{document}